\documentclass[journal]{IEEEtran}
\usepackage{cite}
\usepackage{amsmath,amssymb,amsfonts}
\usepackage{graphicx}
\usepackage{textcomp}
\usepackage{float}
\usepackage{multicol}
\usepackage{amsmath}
\usepackage{amsthm}
\usepackage{amssymb}
\usepackage{amsfonts}
\usepackage{graphicx}
\usepackage{subfigure}
\usepackage{url}
\usepackage{booktabs}
\usepackage{float}
\usepackage{color}
\usepackage{bm}

\newcommand{\col}{\hbox{col}}
\newtheorem{assmp}{\bf Assumption}

\newtheorem{rem}{\bf Remark}

\newtheorem{thm}{\bf Theorem}

\newcommand{\bx}{{\bf x}}
\newcommand{\by}{{\bf y}}

\newcommand{\EQ}{\begin{eqnarray}}
\newcommand{\EN}{\end{eqnarray}}
\newcommand{\EQQ}{\begin{eqnarray*}}
\newcommand{\ENN}{\end{eqnarray*}}

\ifCLASSINFOpdf
\else
\fi

\begin{document}
\title{Distributed Nash Equilibrium Seeking for a Class of Uncertain Nonlinear Systems subject to Bounded Disturbances}

\author{Jie~Huang,~\IEEEmembership{Life Fellow,~IEEE}
	\thanks{This article is a slightly updated version of
		\cite{Huang24}, and the work described in this article
		was supported in part by the Research Grants Council of the Hong Kong Special Administrative Region under grant No. 14201621 and in part by National Natural Science Foundation of China under Project 61973260.}
\thanks{The author is with the Department of Mechanical and Automation
Engineering, The Chinese University of Hong Kong, Hong Kong (e-mail:
jhuang@mae.cuhk.edu.hk.)}}

\maketitle

\begin{abstract}
In this paper, we study the problem of the distributed Nash equilibrium seeking of $N$-player games  over jointly strongly connected switching networks. The action of each player is governed by  a class of uncertain nonlinear systems.
Our approach integrates the consensus algorithm, the distributed estimator over jointly strongly connected switching networks, and some  adaptive control technique. Furthermore, we  also consider the disturbance rejection problem for bounded disturbances with unknown bounds.
A special case of our results gives the solution of the distributed Nash equilibrium seeking for  high-order integrator systems.
\end{abstract}

\begin{IEEEkeywords}
Nash equilibrium seeking, jointly strongly connected switching graphs, nonlinear systems, adaptive control.
\end{IEEEkeywords}

\IEEEpeerreviewmaketitle

\section{Introduction}

\IEEEPARstart{T}he problem of {the} Nash equilibrium seeking for games with multiple players has been well studied for the perfect information case, in \cite{Shamma2004, Basar1999, Facchinei2003, Flam2002, Li1987, G2014}, to name just a few. In practice, not every player can observe the actions of other players and the players have to communicate with each other over a communication network
describing the information exchanges of different players. Such a scenario is called the imperfect information case.
The problem of the Nash equilibrium seeking for the imperfect information case is also called the distributed Nash equilibrium seeking problem.
According to the control systems governing the actions of the players, a game can be classified as single integrator game and  high-order integrator game.
The distributed Nash equilibrium seeking problem was studied first for the single integrator game over static and connected communication networks in \cite{Ye2017, Gadjov2019, Gharesifard2013, Ye2018, Ye201804, Ye2021, Suad2020, Romano2019}, and it was further studied recently in  \cite{He2021}  over jointly strongly connected switching networks. The distributed Nash equilibrium seeking problem  was also studied for high-order integrator  games over
static and connected communication networks in \cite{Romano2019, Romano2020} and  over jointly strongly connected switching networks in  \cite{He2022b, He2024}.
A control system is often subject to some external disturbances caused by, for example, the  insensitivity of the sensors, uncertainty of system dynamics.  The problem of the distributed Nash equilibrium seeking  with disturbance rejection  has also been studied by quite a few papers.
For example, references  \cite{He2022b, Romano2019, Romano2020}  studied the Nash equilibrium seeking for high-order integrator dynamics subject to disturbances generated by a known linear autonomous system called exosystem.

So far, the study on the distributed Nash equilibrium seeking is limited to single-order or high-order integrator systems whose dynamics are exactly known. In the real world,
uncertainty and nonlinearity are ubiquitous.
In this paper, we will further study the distributed Nash equilibrium seeking problem for games whose actions are governed by  a class of uncertain nonlinear systems over jointly strongly connected switching networks. This class of systems contains the high-order integrator systems
as a special case. As the approaches in the  existing references such as  \cite{Ye2017, Gadjov2019, Gharesifard2013, Ye2018, Ye201804, Ye2021, Romano2019, Romano2020} cannot handle
uncertain  nonlinear dynamics, we manage to develop an approach integrating the consensus algorithm, the distributed estimator over jointly strongly connected switching networks, and some adaptive control technique to tackle the problem.
It turns out that such an integrated approach is effective and indeed solve the problem under consideration. It is noted that
the class of nonlinear systems is common in the adaptive control literature  \cite{slotine}.
But,  as pointed out in Remark \ref{rem6}, the distributed Nash equilibrium seeking problem is more complex than the  mere adaptive stabilization  problem of the same system because
the Nash equilibrium is unknown and the communications of the agents are subject to switched networks, which can be disconnected at every time instant.

It is also noted that  \cite{Zhang2019, Zhang2020} studied the disturbance rejection problem with disturbances generated by an uncertain linear exosystem for aggregative games over static, connected and undirected networks, but the approach in  \cite{Zhang2019, Zhang2020} cannot be carried over to solve our problem.


The rest of the paper is organized as follows.
Section II  summarizes basic knowledge for game theory based on \cite{Gadjov2019, G2014},  and
some existing results from  \cite{He2022b, Romano2019}.
Section III presents the main result. 

\indent\textbf{Notation} The notation $\vert \vert x \vert \vert$ denotes 2-norm of a vector $x$ while $\vert \vert P\vert \vert$ is the induced Euclidean norm for a matrix $P$.
$\mathbb{R}^n$ is the $n$-dimensional Euclidean space.
$\mathbb{R}^{m\times n}$ is the set of all $m\times n$ real matrices. For a positive definite symmetric matrix $P \in \mathbb{R}^{n\times n}$, $\lambda_{\min} (P)$ denotes the minimal eigenvalue of $P$. For any time function $x (t)$ and nonnegative integer $r$, $x^{(r)} $ denotes the $r$th derivative of $ x(t)$ with
$x^{(0)} = x$.
For column vectors $a_i, i=1,\cdots,n$,  $\mbox{col} (a_1,\cdots,a_n )= [a_1^T,\cdots,a_n^T  ]^T$. For any matrices $A_1, \cdots, A_n$,
$\mbox{diag}(A_1,...,A_n) =
	\left[\begin{array}{ccc}
A_1&&\\
&\ddots&\\
&&A_n
\end{array}\right]$.
$A \otimes B$ denotes the Kronecker product of any two matrices $A$ and $B$.
$\bm{0}_N$ and $\mathbf{1}_N$ are the $N$-dimensional column vectors with all elements $0$ and $1$, respectively,  and $I_{N}$ represents an $N \times N$ identity matrix.
 We use $\sigma(t)$ to denote a piecewise constant switching signal $\sigma:[0,+\infty)\rightarrow \mathcal{P}=\{1,2,\dots,n_{0}\}$, where $n_{0}$ is a positive integer, and
$\mathcal{P}$ is called a switching index set. We assume that all switching instants $t_0=0 <t_1<t_2,\dots$ satisfy $t_{i+1}-t_{i}\geq \tau_{0}>0$ for some constant $\tau_{0}$ and all $i =  0, 1, 2, \cdots$, where $\tau_{0}$ is called the dwell time. A function $d: [t_0, \infty) \rightarrow  \mathbb{R}^n$ is said to be piecewise continuous if there exists a sequence $\{\tau_j,~ j = 0, 1, \ldots \}$ with a dwell time $\tau >0$ such that $d (t)$ is continuous on each time interval $[\tau_j, \tau_{j+1})$, $j = 0, 1, \ldots$.  Let
$||d||_{\infty} = \sup_{t \geq 0} d (t)$, which is called  the infinity norm of $d$.  $d$ is said to be bounded over $ [0, \infty)$ if $||d||_{\infty} $ is finite.

\section{Preliminaries}
\indent In this section, we introduce the basics of the non-cooperative game theory based on \cite{Gadjov2019, G2014}.

\subsection{Non-cooperative Games}\label{noncooperative}

A non-cooperative game denoted by $\Gamma$ is defined by a triplet as follows:
\EQ \label{game}
\Gamma \triangleq \{\mathcal{V}, f_i, U_i\},
\EN
where $\mathcal{V}$ is the set of $N$ players. For each player $i \in  \mathcal{V}$, the strategy of player $i$ is denoted by  $x_i \in U_i \subset \mathbb{R}^{n_i}$.  Let $\sum_{i=1}^N n_i=n$ and $U =U_1\times U_2 \times \cdots \times U_N \subset \mathbb{R}^n$, which is called the strategy space. Then, $f_i: U \to \mathbb{R}$ is the cost function for player $i$.  Let $\bm{x}=\col(x_1,x_2,\cdots,x_N) \in \mathbb{R}^n$, which is called the strategy vector,
$x_{-i}\triangleq (x_1,x_2, \cdots, x_{i-1},x_{i+1}, \cdots, x_N)$, and $U_{-i} =U_1\times \cdots \times U_{i-1} \times U_{i+1} \cdots \times U_N$.
Then, the goal of each player $i$ is, for all $x_{-i} \in U_{-i}$,  to minimize its cost function $f_i (x_i, x_{-i})$ over $x_i \in U_i$, that is,
\EQ
\mbox{ minimize }  f_i (x_i, x_{-i})~~\mbox{ subject to }~~ x_i \in U_i.
\EN
A strategy vector $\bm{x}^*=(x_i^*,x_{-i}^*) \in U$ is said to be a Nash equilibrium point (NE) if it is such that
\begin{align}
f_i(x_i^*,x_{-i}^*) \leq f_i(x_i,x_{-i}^*),~~  x_i \in U_i,~~ \forall i \in \mathcal{V}.
\end{align}

 To introduce  two standard assumptions,
let $F(\bm{x})= {\col} \left( \nabla_1 f_1 (x_1,x_{-1}), \nabla_2 f_2 (x_2,x_{-2}), \cdots, \nabla_N f_N (x_N,x_{-N})  \right)$
where  $\nabla_i f_i (x_i,x_{-i}) = \frac{\partial f_i(x_i,x_{-i})}{\partial x_i} \in  \mathbb{R}^{n_i}$,
which is called the pseudogradient of $f$.

\begin{assmp} \label{ass3.1}
For all $i \in \mathcal{V}$,  i) $U_i$ is nonempty, closed and convex; ii) the cost function $f_i (x_i, x_{-i})$  is convex and continuously differentiable in  ${x}_i$ for every fixed $x_{-i} \in U_{-i}$; iii) The pseudogradient $F$ is strongly monotone on $U$,
i.e., for some  $\mu >0$,
$$(\bx-\bx')^T (F(\bx)-F(\bx')) \geq \mu ||\bx-\bx'||^2,~~ \forall \bx,~ \bx' \in U.$$
\end{assmp}

\begin{assmp}  \label{ass3.2}
For all $i \in \mathcal{V}$,   $\nabla_i f_i (\bm{x}) $ is   Lipschitz continuous, i.e.,
$||\frac{\partial f_i(\bm{x})}{\partial x_i}  - \frac{\partial f_i(\bm{x}')}{\partial x_i} || \leq \psi_i ||\bm{x}-\bm{x}'||, ~\forall \bm{x}, \bm{x}' \in U$, for some constant $\psi_i >0$.
\end{assmp}

\begin{rem} \label{rem3} The following system
\EQ \label{3.1}
\dot{x}_i=-\nabla_if_i(x_i,x_{-i}), ~~ \forall~ i \in \mathcal{V}
\EN
 is called pseudogradient dynamics \cite{Gadjov2019}, which can be put into the following compact form:
\EQ \label{3.2}
\dot{\bm{x}}=-F(\bm{x}).
\EN
By Theorem 3 of \cite{G2014}, under Parts (i) and (ii) of Assumption 1,  a pure Nash equilibrium $\bx^* \in U$ exists, which  satisfies the following variational inequality:
\EQQ
(\bm{x}-\bm{x}^*)^T F(\bm{x}^*) \geq 0, ~ \bm{x} \in U.
\ENN
Under Assumption 1 with $U =  \mathbb{R}^n$,
the game $\Gamma$ has a unique Nash equilibrium  which  satisfies
\EQQ
F(\bx^*)={\bf 0}_n.
\ENN
Moreover, by Lemma 2 of  \cite{Gadjov2019}, under Assumption 1,   the NE $\bx^*$ of \eqref{3.2} is {globally exponentially stable}. Like in \cite{Ye2017, He2022b, Romano2019, Romano2020},
in what follows, we focus on the global case. Thus, it is assumed that  $U =  \mathbb{R}^n$.
\end{rem}

\subsection{Games with Uncertain Nonlinear  Dynamics}
Since system \eqref{3.1} can be viewed as the closed-loop system of the single-integrator system $\dot{x}_i = u_i$ under the state feedback control
$u_i = -\nabla_if_i(x_i,x_{-i})$, we call the game defined in \eqref{game} as the single-integrator game.
In what follows, we consider
the games whose players' actions $x_i$ are governed by  the following uncertain nonlinear  control systems:
\begin{equation} \label{eq14x}
{x}^{(r_i)}_i +   g_{i} (\xi_i, t) \theta_{i} = u_i + d_i,  ~ \forall~ i \in \mathcal{V}
\end{equation}
where $r_i \geq 1$, $x_i, u_i \in {\mathbb{R}}^{n_i}$, $\xi_i  = \col (x_i, \dot{x}_i, \cdots, {x}_i^{(r_i-1 )} )$,  $g_{i}:\mathbb{R}^{n_i r_i }\times [0,+\infty)\rightarrow\mathbb{R}^{n_i \times m_i}$ are known functions satisfying locally Lipschitz condition with respect to $\xi_{i}$ uniformly in $t$ and continuous in $t$,  $\theta_{i}\in\mathbb{R}^{m_i}$ are unknown constant parameter vectors,
and $d_{i}: [0,+\infty) \rightarrow\mathbb{R}^{n_i} $ are piecewise continuous bounded time functions with the bounds unknown, i.e., $\|d_i (t)\| \leq D_i$ for some unknown positive numbers $D_i$ and all $ t\geq 0$.

\begin{rem}
The disturbance rejection problem have been studied in several papers when  $d_i$ are generated by the following systems: $\forall i \in \mathcal{V}$,
\EQ \label{dis}
 \dot{v}_i=S_i v_i, ~~   d_i=D_i v_i
\EN
with $v_i \in \mathbb{R}^{q_i}$, and $d_i \in \mathbb{R}^{n_i}$.

For example, the Nash equilibrium seeking for the high-order integrator system, which is the special case of system \eqref{eq14x} with $\theta_i = 0$, was studied in \cite{Romano2019, Romano2020} over static and connected  networks and in \cite{He2022b}  over jointly strongly connected switching networks. More recently,
the Nash equilibrium seeking for the high-order  integrator system subject to bounded disturbances with unknown bounds was considered in \cite{acc23}.
\end{rem}


\section{Main Result}

In this section, we study the Nash equilibrium seeking for games with the nonlinear dynamics (\ref{eq14x}) over jointly strongly connected switching graphs.

The state space realization of \eqref{eq14x} is as follows: $\forall i\in \mathcal{V}$,
\begin{equation} \label{eq14}
\dot{\xi}_i = A_i \xi_i + B_i (u_i+d_i -g_{i} (\xi_i, t) \theta_{i}),
\end{equation}
where $\xi_i = \col (x_i, \dot{x}_i, \cdots, {x}_i^{(r_i-1 )} )$,
\EQ \label{eq213}
A_i &=&
\begin{bmatrix}
0 & 1& \cdots & 0 \\
\vdots & \vdots & \ddots & \vdots \\
0 & 0 & \cdots & 1 \\
0 & 0 & \cdots & 0
\end{bmatrix} \otimes I_{n_i},~ B_i=
\begin{bmatrix}
0 \\
\vdots \\
0 \\
1
\end{bmatrix} \otimes I_{n_i}, \nonumber \\
C_i &=&
\begin{bmatrix}
1 & 0 & \cdots & 0
\end{bmatrix}\otimes I_{n_i}.
\EN

Like in \cite{Romano2019, Romano2020, He2022b},  define a fictitious output
\begin{equation} \label{eq36}
\gamma_i = \sum_{k=0}^{r_i-2} {c_k^i} x_i^{(k)} + x^{(r_i - 1)}_i
\end{equation}
where $c_k^i$ are such that the  polynomials $a_i(s) = s^{(r_i-1)} + c^i_{(r_i-2)} s^{(r_i-2)} +\cdots + c_1^i s + c_0^i$ are Hurwitz with $c_0^i =1$.

Let  $x_i^s= \col (\dot{x}_i, \cdots, {x}_i^{(r_i -1)}) \in \mathbb{R}^{n_i (r_i-1)}$.  Then  performing the coordinate transformation $\xi_i \mapsto (\gamma_i, x_i^s)$ gives:    $\forall i \in \mathcal{V}$,
\begin{subequations} \label{eq34y}
\begin{align}
    \dot{\gamma}_i &= u_i + d_i - g_{i}(\xi_i, t)\theta_{i}   + K^s_i x_i^s \label{eq34ya}\\
   \dot{x}_i^s&=A_i^s x_i^s + B_i^s  (u_i +d_i -    g_{i}(\xi_i, t)\theta_{i} ) \label{eq34yc}
\end{align}
\end{subequations}
where
$K_i^s = \begin{bmatrix}
c_0^i & c_1^i &\cdots & c_{r_i-2}^i
\end{bmatrix}\otimes I_{n_i}$,
\EQQ \label{Ais}
A_i^s &=& \begin{bmatrix}
0 & 1 & \cdots & 0 \\
\vdots & \vdots & \ddots & \vdots \\
0 & 0 & \cdots & 1 \\
0 & 0 &\cdots & 0
\end{bmatrix} \otimes I_{n_i},~
B_i^s=\begin{bmatrix}
0 \\
\vdots \\
0 \\
1
\end{bmatrix} \otimes I_{n_i}.
\ENN

\subsection{Perfect Information}
For the sake of the better readability of the paper, let us first consider the perfect information case where the control $u_i$ can access the state $\xi_j$ of
(\ref{eq14x}) for all $j \in {\cal V}$, $\theta_i$ are all known, and $d_i = 0$. In this case, we allow our control law to be a full information control law as follows:
\begin{equation} \label{eq20u}
u_i=  h_i (\xi_1, \cdots, \xi_N),~~  i \in \mathcal{V}
\end{equation}
where $h_i$ are globally defined functions to be designed.

Let $\gamma_{-i} =  \col (\gamma_{1}, \cdots, \gamma_{i-1},\gamma_{i+1}, \cdots, \gamma_{N})$. Consider the following control law: $\forall i \in \mathcal{V}$,

\begin{equation} \label{pconlaw}
u_i=-\nabla_if_i(\gamma_i, \gamma_{-i}) +   g_{i} (\xi_i, t)\theta_{i} -  K^s_i x_i^s.
\end{equation}
Then, under the control law \eqref{pconlaw},   the closed-loop system with $d_i = 0$ is as follows:    $\forall i, j \in \mathcal{V}$,
\begin{subequations} \label{eq34aaa}
\begin{align}
    \dot{\gamma}_i &=-\nabla_if_i(\gamma_i, \gamma_{-i}) \label{eq34a}\\
    \dot{x}_i^s&=A_i^K x_i^s -B_i^s  \nabla_if_i(\gamma_i, \gamma_{-i})  \label{eq34c}
\end{align}
\end{subequations}
where
\EQQ \label{Ais}
A_i^K &=& \begin{bmatrix}
0 & 1 & \cdots & 0 \\
\vdots & \vdots & \ddots & \vdots \\
0 & 0 & \cdots & 1 \\
-c_0^i & -c_1^i &\cdots & -c_{r_i-2}^i
\end{bmatrix}\otimes I_{n_i}
\ENN
are all Hurwitz.
Let  $\bm{\gamma} = \mbox{col}( \gamma_1, \cdots, \gamma_N) $ and $\bm{x}^s=\mbox{col}(x_1^s,\cdots,x_N^s)$. Then the compact form of system (\ref{eq34aaa}) is as follows:
\begin{subequations} \label{eq38xyd}
\begin{align}
    \dot{\bm{\gamma}}&=-  F (\bm{\gamma})  \label{eq38xyda}  \\
    \dot{\bm{x}}^s&= A^K \bm{x}^s-B F (\bm{\gamma}) \label{eq38xydb}
\end{align}
\end{subequations}
where $A^K=\mbox{diag}(A_1^K,\cdots,A_N^K)$ and $B=\mbox{diag}(B_1^s,\cdots,B_N^s)$.

\begin{thm} \label{eq34}
Under Assumption \ref{ass3.1}, the unique equilibrium $(\bm{x}^*,\bm{0}_{\sum_{i=1}^{N} n_i (r_i-1)})$ of the closed-loop system \eqref{eq38xyd}  is globally exponentially stable.
\end{thm}

\begin{proof}
The conclusion of this theorem is obvious if one notes that $F (\bm{\gamma})$ tends to the origin exponentially as $\bm{\gamma}$ tends to $\bm{x}^*$ exponentially.
Nevertheless, to see the role of Assumption \ref{ass3.1}, we provide a detailed proof below.
By Remark \ref{rem3}, under Assumption \ref{ass3.1},  \eqref{eq38xyda} has a unique Nash equilibrium $\bx^*$.
Define a Lyapunov function candidate for \eqref{eq38xyda}
as follows:
\EQ  \label{V}
V (\bm{\gamma} - \bm{x}^*)= \frac{1}{2} (\bm{\gamma}- \bx^*)^T (\bm{\gamma}- \bx^*).
\EN
Then,   the derivative of ${V}$ along the solution of \eqref{eq38xyda}  satisfies
\EQ \label{314xy}
 \dot{V}  = - (\bm{\gamma} - \bm{x}^*)^T F (\bm{\gamma}).
 \EN
Since $F(\bm{x}^*) = 0$,  by Assumption \ref{ass3.1}, we have
 \EQ \label{315b}
 \begin{split}
 (\bm{\gamma} - \bm{x}^*)^T F (\bm{\gamma})&=(\bm{\gamma} - \bm{x}^*)^T (F(\bm{\gamma})- F(\bm{x}^*)) \\
 & \geq \mu || \bm{\gamma} - \bm{x}^*||^2.
 \end{split}
 \EN
Substituting \eqref{315b}  into \eqref{314xy} gives
\EQ \label{eq323}
\dot{V} \leq - \mu || \gamma - \bm{x}^*||^2.
\EN
Thus, $\lim_{t \rightarrow \infty}  \gamma (t) = \bm{x}^*$ exponentially.
Now, consider equation (\ref{eq36}), which can be viewed as a stable linear differential equation in $x_i$ with the input $\bm{\gamma}_i (t)$ satisfying  $\lim_{t \rightarrow \infty}  \bm{\gamma}_i (t) =  \bm{x}^{*}_i$ exponentially. Thus,
$\lim_{t \rightarrow \infty} \bm{x}^s (t) =  0$ exponentially.
\end{proof}

\subsection{Imperfect Information Case without Disturbances}

We now further consider the imperfect information case without disturbances.  Like in \cite{He2022b},  corresponding to the game described in (\ref{game}),
we can define a switching graph\footnote{See  \cite{cshbook} or \cite{God2001} for a summary of graph.} $\mathcal{G}_{\sigma (t)} =(\mathcal{V},\mathcal{E}_{\sigma (t)})$ with $\mathcal {V}=\{1,\dots,N\}$ and $\mathcal{E}_{\sigma (t)} \subset\mathcal{V}\times\mathcal{V}$ for all $t \geq 0$. For any $t \geq 0$, $\mathcal{E}_{\sigma (t)}$ contains an edge $(j, i)$ if and only if the player $i$ is able to observe the state $\xi_{j}$ of player $j$ at time $t$.

We assume all the players can communicate with each other over  a communication graph
satisfying the following assumption:
\begin{assmp} \label{ass3.7}
There exists a subsequence $\{i_k\}$ of $\{i:i=0,1,2,\dots\}$ with
$t_{i_{k+1}}-t_{i_k}< \nu$ for some positive number $\nu$ such that the union graph
${\mathcal{G}}_{\sigma ([t_{i_k},t_{i_{k+1}}))}$ is strongly connected.
\end{assmp}

\begin{rem}
As in \cite{He2021} and \cite{He2022b}, we say a switching graph ${\mathcal{G}}_{\sigma (t)}$ satisfying Assumption \ref{ass3.7} is jointly strongly  connected.
Under Assumption \ref{ass3.7}, the graph can be disconnected at every time instant. Thus,
the approaches in \cite{Romano2019, Romano2020, Ye2017, Ye2018, Ye201804, Ye2021} do not apply to this case.
\end{rem}

Let us first summarize the main result of \cite{He2021} which studied the Nash equilibrium seeking of single-integrator systems.
The control law in \cite{He2021}  is as follows: $\forall i, j \in \mathcal{V}$,
\begin{subequations} \label{conlaw}
\begin{align}
u_i&=-\delta {k}_i \nabla_if_i(\bm{y}_i),  \label{conlawa}\\
 \dot{y}_{ij}&=-\left(\sum_{k=1}^N a_{ik}(t)(y_{ij}-y_{kj})+a_{ij}(t)(y_{ij}- x_j) \right),  \label{conlawab}
\end{align}
\end{subequations}
where ${k}_i$ are fixed positive numbers, $\delta >0$ is some positive number,  $a_{ij} (t)$ are the elements of the adjacency matrix of the switching graph ${\mathcal{G}}_{\sigma (t)}$, $y_{ij} \in \mathbb{R}^{n_j}$ is interpreted as the estimate of  $x_j$ by player $i$, and $\by_i= \mbox{col~} (y_{i1},y_{i2},\cdots,y_{iN}) \in \mathbb{R}^{n} $ is the estimate of the strategy vector $\bx$ by player $i$.

Under the control law \eqref{conlaw},   the closed-loop system is as follows:    $\forall i, j \in \mathcal{V}$,
\begin{subequations} \label{eq34}
\begin{align}
    \dot{x}_i &=-\delta {k}_i \nabla_if_i(\bm{y}_i) \label{eq34a}\\
     \dot{y}_{ij}&=-\left(\sum_{k=1}^N a_{ik}(t)(y_{ij}-y_{kj})+a_{ij}(t)(y_{ij}-x_j) \right).     \label{eq34b}
\end{align}
\end{subequations}

Let ${\mathcal L}_{\sigma(t)}$ denote the Laplacian of the switching graph ${\mathcal{G}}_{\sigma (t)}$,
 $\by=\mbox{col~} (\by_1,\cdots,\by_N) \in \mathbb{R}^{N n}$, and
 \EQQ
&B_{\sigma(t)} = \mbox{diag}  (  a_{11} (t) I_{n_1},  \cdots, a_{1N}(t)I_{n_N}, a_{21}(t)I_{n_1},\\
& \cdots,
 a_{2N}(t)I_{n_N}, \cdots, a_{N1}(t)I_{n_1}, \cdots, a_{NN}(t)I_{n_N}  ).
\ENN
Then system (\ref{eq34}) takes the following compact form:
\begin{subequations} \label{eq38xx}
\begin{align}
    \dot{\bm{x}}&=-\delta \bm{k} H(\bm{y})  \label{eq38xxa}  \\
    \dot{\bm{y}}&=-({\mathcal L}_{\sigma(t)} \otimes I_{N}+B_{\sigma(t)}) \bm{y} + B_{\sigma(t)} (\bm{1}_N \otimes \bm{x}) \label{eq38xxb}
\end{align}
\end{subequations}
where $\bm{k} = \mbox{diag} (k_1 I_{n_1}, \cdots, k_N I_{n_N})$ and
 $H (\bm{y}) = \col (\nabla_1 f_1 (\bm{y}_1), \nabla_2 f_2 (\bm{y}_2) \cdots,  \nabla_N f_N (\bm{y}_N))$.

Then Lemma 4.1 of \cite{He2021} can be rephrased as follows:

\begin{thm} \label{th3.2}
 Under Assumptions \ref{ass3.1} to \ref{ass3.7},  there exists $\delta^*>0$ such that,  for $0<\delta<\delta^*$,
 the
 equilibrium  $(\bm{x}^*, {\bf 1}_N \otimes \bm{x}^*)$ of system (\ref{eq38xx})  is globally exponentially stable.
\end{thm}

\begin{rem}  \label{rem6}
Let
$\tilde{\bm{y}}=\bm{y}-\bm{1}_{N} \otimes \bm{x}$.
Then,
\EQ  \label{10}
\begin{split}
    \dot{\tilde{\bm{y}}}&=\dot{\bm{y}}-\bm{1}_N \otimes \dot{\bm{x}} \\
    &=-({\mathcal L}_{\sigma(t)}\otimes I_{n}+B_{\sigma(t)})\tilde{\bm{y}}+\bm{1}_N \otimes  (\delta \bm{k} H(\bm{y})).
\end{split}
\EN

Thus, system \eqref{eq38xx} is converted to the following form:
\begin{subequations} \label{eq38y}
\begin{align}
    \dot{\bm{x}}&=-\delta \bm{k} H(\bm{y})  \label{eq38ya}  \\
   \dot{\tilde{\bm{y}}}&=-({\mathcal L}_{\sigma(t)}\otimes I_{n}+B_{\sigma(t)})\tilde{\bm{y}}+\bm{1}_N \otimes  (\delta \bm{k} H(\bm{y})). \label{eq38yb}
    \end{align}
\end{subequations}

By Lemma 4.1 of \cite{He2021}, under Assumption \ref{ass3.7}, the origin of the following system:
\EQ \label{linear}
\dot{\tilde{\bm{y}}}=-({\mathcal L}_{\sigma(t)} \otimes I_{n}+B_{\sigma(t)}) \tilde{\bm{y}}
\EN
is exponentially stable, and,  for any  constant positive definite matrix $Q \in \mathbb{R}^{N n \times N n}$, there exists a symmetric matrix $P(t) \in \mathbb{R}^{N n \times N n}$ which is bounded and continuous for all $t \geq 0$ such that   $||P(t)|| \leq p$ for some positive constant $p$ for all $t \geq 0$,
\EQ \label{312}
c_1 I_{Nn} \leq  P(t) \leq c_2 I_{Nn}
\EN
for some positive constants $c_1$ and $c_2$, and, for $t \in [t_j, t_{j+1}), j=0, 1,2,\cdots $
\EQ \label{313}
\begin{split}
\dot{P}(t)
&= P (t) ({\mathcal L}_{\sigma(t)} \otimes I_{n} +B_{\sigma(t)}) \\
&+({\mathcal L}_{\sigma(t)} \otimes I_{n}+B_{\sigma(t)})^T P(t) -Q.
\end{split}
\EN

Moreover,  let  $\bm{w}= \col ((\bm{x}- \bm{x}^*),\tilde{\bm{y}})$ and
\EQ \label{lfc}
\begin{aligned}
    V_1 (\bm{w}, t) &=\frac{1}{2} (\bm{x}- \bm{x}^*)^T \bm{k}^{-1} (\bm{x}- \bm{x}^*)  +\tilde{\bm{y}}^T P (t) \tilde{\bm{y}}.
    \end{aligned}
\EN
Then, along the  trajectory of (\ref{eq38y}), $\dot{V}_1$ satisfies
\EQ	
\begin{split}
\dot{V}_1 &\leq  -\delta  \left[ \begin{array}{cc}
  ||\bm{x}-\bm{x^*}|| & ||\bar{\bm{y}}||
 \end{array}
 \right] Y \left[
 \begin{array}{c}
  ||\bm{x}-\bm{x^*}|| \\
 ||\bar{\bm{y}}||
 \end{array}
 \right]
    \end{split}
\EN
where
\EQ \nonumber
Y=
\left[
 \begin{array}{cc}
  \mu  & -\frac{\psi}{2}-k \psi p  \\
 -\frac{\psi}{2}-k \psi p  &  \frac{\lambda_{min}(Q)}{\delta}-2k \psi p
 \end{array}
 \right]
\EN
with $k = ||\bm{k}||$ and $\psi = \sqrt{\psi_1^2+ \ldots+ \psi_N^2}$.
Let $ \delta^* = \frac{4 \mu \lambda_{min}(Q)}{ \psi^2 + 4 k^2 \psi^2 p^2 + 4 \mu k p \psi^2  + 8 \mu k \psi p }$. Then, for all $0 < \delta < \delta^*$, the matrix $Y$ is positive definite. Hence,
$\dot{V} \leq -\delta \lambda_{min}(Y)||\bm{w}||^2$. Thus,  $(\bm{x}^*, \bm{0}_{Nn})$  is a globally exponentially stable equilibrium of  system (\ref{eq38y}),  which implies
  $(\bm{x}^*, {\bf 1}_N \otimes \bm{x}^*)$  is a globally exponentially stable equilibrium of  system (\ref{eq38xx}).
\end{rem}
%

Now, we  propose our control law as follows: $\forall i,j\in \mathcal{V}$,
\begin{subequations} \label{equ}
\begin{align}
{u}_i &=- \delta {k}_i \nabla_if_i(\bm{z_i}) + g_i (\xi_i, t) \hat{\theta}_i  - K^s_i x_i^s + \kappa_i (\hat{\gamma}_i-\gamma_i)  \label{equ1} \\
\dot{\hat{\gamma}}_i &= - \delta {k}_i \nabla_if_i(\bm{z_i}) \label{equ2} \\
\dot{z}_{ij} &=-\left(\sum_{k=1}^N a_{ik}(t)(z_{ij}-z_{kj})+a_{ij}(t)(z_{ij}-\hat{\gamma}_j) \right) \label{equ4} \\
\dot{\hat{\theta}}_i&= \Lambda_i g_i^T (\xi_i, t) (\hat{\gamma}_i-\gamma_i) \label{equ5}
\end{align}
\end{subequations}
where $\kappa_i$ are  positive numbers, $z_{ij} \in \mathbb{R}^{n_j}$,  $\bm{z_i}=\col (z_{i1},z_{i2},\cdots,z_{iN}) \in \mathbb{R}^n$, $\Lambda_i \in \mathbb{R}^{m_i \times m_i} $ are symmetric and positive definite matrices, and $\hat{\theta}_i$ is  the estimate of $\theta_i$.

Under the control law \eqref{equ}, the closed-loop system composed of \eqref{eq34y} with $d_i =0$ and \eqref{equ} is as follows: $\forall i,j\in \mathcal{V}$,
\begin{subequations} \label{equxy}
\begin{align}
\dot{\gamma}_i &=- \delta {k}_i \nabla_if_i(\bm{z_i})+g_i (\xi_i, t) \tilde{\theta}_i+ \kappa_i (\hat{\gamma}_i-\gamma_i)  \label{equxy1} \\
\dot{x}_i^s&=A_i^K x_i^s+B_i^s  (- \delta {k}_i \nabla_if_i(\bm{z_i}) +g_i \tilde{\theta}_i
+ \kappa_i (\hat{\gamma}_i-\gamma_i) ) \label{equxy2} \\
\dot{\hat{\gamma}}_i &= - \delta {k}_i \nabla_if_i(\bm{z_i}) \label{equxy3} \\
\dot{z}_{ij} &=-\left(\sum_{k=1}^N a_{ik}(t)(z_{ij}-z_{kj})+a_{ij}(t)(z_{ij}-\hat{\gamma}_j) \right)  \label{equxy5} \\
\dot{\tilde{\theta}}_i&=  \Lambda_i g_i^T (\xi_i,t) (\hat{\gamma}_i-\gamma_i) \label{equxy6}
\end{align}
\end{subequations}
where $\tilde{\theta}_i=\hat{\theta}_i- \theta_i$.

Let $\bm{\hat{\gamma}}= \col (\hat{\gamma}_1, \hat{\gamma}_2, \cdots, \hat{\gamma}_N)  $, $\bm{\xi} =\mbox{col} (\xi_1, \cdots, {\xi}_N)$, $\bm{g} (\bm{\xi},t)=\mbox{diag}({g_1}({\xi}_1,t),\cdots,{g_N} ({\xi}_N,t)) $,
$\bm{z}=\mbox{col}(\bm{z_1},\cdots,\bm{z_N}) $, $\bm{\Lambda}=\mbox{diag}(\Lambda_1, \cdots, \Lambda_N)$,   $\bm{\hat{\theta}}=\mbox{col}(\hat{\theta}_1, \cdots,\hat{\theta}_N)$, $\bm{\tilde{\theta}}=\mbox{col}(\tilde{\theta}_1, \cdots,\tilde{\theta}_N)$, and $\kappa = \mbox{diag}(\kappa_1 I_{n_1}, \cdots, \kappa_N I_{n_N})$.
Then, the concatenated form of the closed-loop system (\ref{equxy})  is as follows:
\begin{subequations} \label{equc}
\begin{align}
\dot{\bm{\gamma}}&=-\delta \bm{k} H(\bm{z}) + \bm{g} (\bm{\xi},t)\bm{\tilde{\theta}}+ \kappa(\bm{\hat{\gamma}}-\bm{\gamma}) \label{equc1}  \\
\dot{\bm{x}}^s&=A^K \bm{x}^s+B(-\delta \bm{k}H(\bm{z}) +\bm{g} (\bm{\xi}, t)\bm{\tilde{\theta}}+ \kappa(\bm{\hat{\gamma}}-\bm{\gamma}))    \label{equc2}   \\
\dot{\bm{\hat{\gamma}}}&=-\delta \bm{k} H(\bm{z}) \label{equc3}  \\
    \dot{\bm{z}}&=-({\mathcal L}_{\sigma(t)} \otimes I_{n}+B_{\sigma(t)}) \bm{z}+B_{\sigma(t)} (\bm{1}_N \otimes \bm{\hat{\gamma}})  \label{equc5}   \\
  \dot{\bm{\tilde{\theta}}}&= \bm{\Lambda} \bm{g}^T (\bm{\xi}, t) (\bm{\hat{\gamma}}-\bm{\gamma}). \label{equc6}
     \end{align}
\end{subequations}

We need one more assumption as follows.

 \begin{assmp}\label{ass6}
$\forall~ i \in \mathcal{V}$, $||g_i (\xi_i, t)|| \leq \phi_i (\xi_i) $ for some globally defined functions $\phi_i (\xi_i)$.
\end{assmp}

\begin{rem}
Assumption \ref{ass6} is satisfied automatically if $g_i (\xi_i, t)$ are all independent of $t$.
\end{rem}

We now state our main result as follows.

\begin{thm} \label{thm3.2}
Under Assumptions \ref{ass3.1} to  \ref{ass6},
there exist $\delta^*>0$ such that,  for any  $0<\delta<\delta^*$, and any initial condition,   the solution of the closed-loop system \eqref{equc}  is bounded, and
\begin{subequations} \label{eq38v}
\begin{align}
\lim_{t \rightarrow \infty} \bm{\gamma} (t) &= \bm{x}^*  \label{eq38va} \\
\lim_{t \rightarrow \infty} \bm{x}^s (t) &= \bm{0}_{\sum_{i=1}^N n_i (r_i-1)}  \label{eq38vb}  \\
\lim_{t \rightarrow \infty} \hat{\bm{\gamma}} (t) &= \bm{x}^*  \label{eq38vc} \\
\lim_{t \rightarrow \infty} {\bm{z}} (t) &= {\bf 1} \otimes \bm{x}^*.   \label{eq38ve}
\end{align}
\end{subequations}
As a result, $\lim_{t \rightarrow \infty} \bm{x} (t) = \bm{x}^*$.
\end{thm}

\begin{proof}
Let $\bm{\tilde{\gamma}}=\bm{\hat{\gamma}}-\bm{\gamma}$. Then,
\EQ  \label{32_1}
\begin{aligned}
\dot{\bm{\tilde{\gamma}}}&=\dot{\bm{\hat{\gamma}}}-\dot{\bm{\gamma}} = -\bm{g}  (\bm{\xi},  t) \bm{\tilde{\theta}}- \kappa \bm{\tilde{\gamma}}.
\end{aligned}
\EN

Thus, the closed-loop system \eqref{equc} is converted to the following form:
\begin{subequations} \label{eq38}
\begin{align}
\dot{\bm{\tilde{\gamma}}}&=-\bm{g} (\bm{\xi},  t) \bm{\tilde{\theta}}-  \kappa \bm{\tilde{\gamma}} \label{eq38a}  \\
\dot{\bm{x}}^s&=A^K \bm{x}^s+B(-\delta \bm{k}H(\bm{z}) +\bm{g} (\bm{\xi}, t)\bm{\tilde{\theta}}+ \kappa \bm{\tilde{\gamma}} )  \label{eq38b}  \\
\dot{\bm{\hat{\gamma}}}&=-\delta \bm{k} H(\bm{z}) \label{eq38c}  \\
    \dot{\bm{z}}&=-({\mathcal L}_{\sigma(t)} \otimes I_{n}+B_{\sigma(t)}) \bm{z}+B_{\sigma(t)} (\bm{1}_N \otimes \bm{\hat{\gamma}})  \label{eq38e} \\
    \dot{\bm{\tilde{\theta}}}&=  \bm{\Lambda} \bm{g}^T (\bm{\xi},  t)  \bm{\tilde{\gamma}}.
   \end{align}
\end{subequations}
By Theorem \ref{th3.2}, there exist $\delta^*>0$ such that,  for any $0<\delta<\delta^*$, $(\bm{x}^*,   {\bf 1}_N \otimes \bm{x}^*)$ is the globally exponentially stable equilibrium of  subsystems (\ref{eq38c}) and (\ref{eq38e}).
Thus, (\ref{eq38vc}) and (\ref{eq38ve}) hold.
We only need to show that $\bm{\tilde{\theta}}$ is bounded and (\ref{eq38va}) and \eqref{eq38vb} hold.
For this purpose, consider the following system:
\begin{subequations} \label{eq39}
\begin{align}
\dot{\bm{\tilde{\gamma}}}&=  - \bm{g} (\bm{\xi},  t) \bm{\tilde{\theta}}-\kappa \bm{\tilde{\gamma}} \label{eq39a}  \\
  \dot{\bm{\tilde{\theta}}}&= \bm{\Lambda} \bm{g}^T (\bm{\xi},  t)  \bm{\tilde{\gamma}}. \label{equ39c}
     \end{align}
\end{subequations}

Let $\bar{\bm{z}}= \col (\bm{\tilde{\gamma}} , \bm{\tilde{\theta}})$. Choose the Lyapunov function candidate for \eqref{eq39} as follows:
\EQ \label{lya}
\begin{aligned}
 V_2 (\bar{\bm{z}})=& \frac{1}{2} \bm{\tilde{\gamma}}^T  \bm{\tilde{\gamma}} + \frac{1}{2}  \bm{\tilde{\theta}}^T \bm{\Lambda}^{-1} \bm{\tilde{\theta}}.
 \end{aligned}
 \EN

\par  Then, the derivative of \eqref{lya} along the solution of \eqref{eq39} satisfies
\EQ \label{314}
\begin{aligned}
 \dot{V}_2 (\bar{\bm{z}}) &= - \bm{\tilde{\gamma}}^T (\bm{g} (\bm{\xi}, t) \bm{\tilde{\theta}}+\kappa \bm{\tilde{\gamma}}) +
\bm{\tilde{\theta}}^T \bm{g}^T (\bm{\xi}, t) \bm{\tilde{\gamma}} \notag \\
 & = - \bm{\tilde{\gamma}}^T \kappa \bm{\tilde{\gamma}}.
\end{aligned}
\EN
Thus,  $\bm{\tilde{\gamma}} $, and $\bm{\tilde{\theta}}$ are both bounded.
Since \eqref{eq38vc} implies  $\bm{\hat{\gamma}} (t) $ is bounded,  $\bm{{\gamma}} (t) $   is also bounded.
Since equation (\ref{eq36}) can be viewed as a stable linear differential equation in $x_i$ with an bounded input $\bm{\gamma}_i (t)$,
$\xi_i$ are all bounded. Thus, $\bm{\xi}$ is bounded.

Since  $\ddot{V}_2 (\bar{\bm{z}})  =  - 2 \bm{\tilde{\gamma}}^T \kappa \dot{\bm{\tilde{\gamma}}}
= 2 \bm{\tilde{\gamma}}^T \kappa (\bm{g} (\bm{\xi},  t) \bm{\tilde{\theta}}+\kappa \bm{\tilde{\gamma}})$, and $ \bm{g} (\bm{\xi},  t)$ is bounded over  $[0, \infty)$ by Assumption \ref{ass6}, $\ddot{V}_2 (\bar{\bm{z}})$ is bounded. Thus,  $\dot{V}_2 (\bar{\bm{z}})$ is uniformly continuous.
 By Babalat's Lemma \cite{Khalil},  $\lim_{t \rightarrow \infty} V_2 (\bar{\bm{z}} (t)) = 0$,  which implies
  $\lim_{t \rightarrow \infty} \bm{\tilde{\gamma}}  (t) =  \bm{0}_{n} $. This fact together with the fact that $\lim_{t \rightarrow \infty} \bm{\hat{\gamma}} (t) =  \bm{x}^{*} $   implies that
$\lim_{t \rightarrow \infty} \bm{\gamma} (t) =  \bm{x}^{*}$.

Finally, consider equation (\ref{eq36}) again, which can be viewed as a stable linear differential equation in $x_i$ with the input $\bm{\gamma}_i (t)$ satisfying  $\lim_{t \rightarrow \infty} \bm{\gamma}_i (t) =  \bm{x}^{*}_i$. Thus, we have
$\lim_{t \rightarrow \infty} \bm{x}^s (t) =  \bm{0}_{\sum_{i=1}^N n_i (r_i-1)}$.
Thus,  $\lim_{t \rightarrow \infty} \bm{\gamma}_i (t) = c^i_0 \lim_{t \rightarrow \infty} \bm{x}_i (t) = \bm{x}^{*}_i$, that is,
$\lim_{t \rightarrow \infty} \bm{x} (t) =  \bm{x}^{*}$.
\end{proof}

\begin{rem} \label{rem6}
 The problem is quite different from the adaptive stabilization of the same system which can be controlled by the following control law:
 \begin{subequations}\label{equds}
\begin{align}
{u}_i &= g_i (\xi_i, t) \hat{\theta}_i  - K^s_i x_i^s - \kappa_i \gamma_i  \label{equds1} \\
\dot{\hat{\theta}}_i&= - \Lambda_i g_i^T (\xi_i, t) \gamma_i.\label{equds5}
\end{align}
 \end{subequations}
Comparing \eqref{equds} with  \eqref{equ}, one can see that in order to deal with the unknown Nash equilibrium and the communication constraints, we need to introduce two more equations \eqref{equ2} and \eqref{equ4} to estimate the Nash equilibrium and to overcome the communication constraints. Also, we need to modify
\eqref{equds1}  into the form \eqref{equ1} to induce a closed-loop system of the form \eqref{equc} so that Theorem \ref{th3.2} can be applied to \eqref{equc3} and \eqref{equc5}.
\end{rem}

\subsection{Imperfect Information Case with Disturbances}

To deal with
the disturbances,  let $\hat{D}_i$ be the estimates of the upper bounds  $D_i$ of $d_i$.
Then,  our control law is as follows: for all $i,j\in \mathcal{V}$,
\begin{subequations} \label{equd}
\begin{align}
{u}_i &=- \delta {k}_i \nabla_if_i(\bm{z_i}) + g_i (\xi_i, t) \hat{\theta}_i  - K^s_i x_i^s  \notag \\
&+ \kappa_i (\hat{\gamma}_i-\gamma_i)  + sgn(\hat{\gamma}_i- \gamma_i) \hat{D}_i  \label{equd1} \\
\dot{\hat{\gamma}}_i &= - \delta {k}_i \nabla_if_i(\bm{z_i}) \label{equd2} \\
\dot{z}_{ij} &=-\left(\sum_{k=1}^N a_{ik}(t)(z_{ij}-z_{kj})+a_{ij}(t)(z_{ij}-\hat{\gamma}_j) \right) \label{equd4} \\
\dot{\hat{\theta}}_i&=   \Lambda_i g_i^T (\xi_i,\eta_i) ( \hat{\gamma}_i- \gamma_i )\label{equd5} \\
\dot{\hat{D}}_{i}&= ( \hat{\gamma}_i - \gamma_{i})^T sgn (\hat{\gamma}_i- \gamma_i)  \label{equd6}
\end{align}
\end{subequations}
where, for any scalar $x$,  the function $sgn (\cdot)$ is defined as follows:
\begin{align}
sgn(x)&=\begin{cases}
1,\quad &x> 0 \\
0,\quad &x= 0 \\
-1,\quad &x<0
\end{cases} \label{sign_fun0}
\end{align}
and, for any vector $x = \col (x_1, \cdots, x_N) \in \mathbb{R}^N$, $sgn (x) = \col (sgn (x_1), \cdots, sgn (x_N))$.

Under the control law \eqref{equd}, the closed-loop system composed of \eqref{eq34y} and \eqref{equd} is as follows: $\forall i,j\in \mathcal{V}$,
\begin{subequations}\label{eq58d}
\begin{align}
\dot{\gamma}_i &=- \delta {k}_i \nabla_if_i(\bm{z_i})+g_i (\xi_i,t)  \tilde{\theta}_i + \kappa_i (\hat{\gamma}_i-\gamma_i) \notag \\
&+sgn( \hat{\gamma}_i - \gamma_i) \hat{D}_i  +d_i   \label{eq58d1} \\
\dot{x}_i^s&=A_i^K x_i^s+B_i^s  (- \delta {k}_i \nabla_if_i(\bm{z_i}) +g_i (\xi_i, t) \tilde{\theta}_i \notag \\
&+\kappa_i (\hat{\gamma}_i-\gamma_i)) +sgn( \hat{\gamma}_i - \gamma_i) \hat{D}_i  +d_i )     \label{equ58d2} \\
\dot{\hat{\gamma}}_i &= - \delta {k}_i \nabla_if_i(\bm{z_i}) \label{equ58d3} \\
\dot{z}_{ij} &=-\left(\sum_{k=1}^N a_{ik}(t)(z_{ij}-z_{kj})+a_{ij}(t)(z_{ij}-\hat{\gamma}_j) \right)  \label{equ58d5} \\
\dot{\tilde{\theta}}_i&=  \Lambda_i g_i^T (\xi_i,t) (\hat{\gamma}_i - \gamma_i )\label{equ58d6} \\
\dot{\tilde{D}}_{i}&=( \hat{\gamma}_i - \gamma_{i})^T sgn (\hat{\gamma}_i- \gamma_i)   \label{eq58d7}
\end{align}
\end{subequations}
where  $\tilde{D}_i=\hat{D}_i - D_i$.

Let
$\bm{d} =\mbox{col} (d_1, \cdots, d_N)$,
 $\hat{D}=\mbox{col}(\hat{D}_1, \cdots, \hat{D}_N)$,   $D=\mbox{col}(D_1, \cdots, D_N)$,  $\tilde{D}=\mbox{col}(\tilde{D}_1, \cdots, \tilde{D}_N)$,
  and
$ {Sgn} (\bm{\gamma}) = \mbox{diag} (sgn(\gamma_1), \cdots, sgn(\gamma_N))$. Then  the compact form of  the closed-loop system (\ref{eq58d})  is as follows:
\begin{subequations} \label{equcd}
\begin{align}
\dot{\bm{\gamma}}&=-\delta \bm{k} H(\bm{z}) + \bm{g} (\bm{\xi},t)\bm{\tilde{\theta}}+ \kappa(\bm{\hat{\gamma}}-\bm{\gamma})\notag \\
 & + {Sgn} (\bm{\bm{\hat{\gamma}}- \gamma})\hat{ D}  + \bm{d}  \label{equcd1}  \\
\dot{\bm{x}}^s&=A^K \bm{x}^s+B(-\delta \bm{k}H(\bm{z}) +\bm{g} (\bm{\xi}, t)\bm{\tilde{\theta}} \notag \\
&+ \kappa(\bm{\hat{\gamma}}-\bm{\gamma}))   + {Sgn} (\bm{\bm{\hat{\gamma}}- \gamma})\hat{ D}  + \bm{d}) \label{equcd2}  \\
\dot{\bm{\hat{\gamma}}}&=-\delta \bm{k} H(\bm{z}) \label{equcd3}  \\
    \dot{\bm{z}}&=-({\mathcal L}_{\sigma(t)} \otimes I_{n}+B_{\sigma(t)}) \bm{z}+B_{\sigma(t)} (\bm{1}_N \otimes \bm{\hat{\gamma}})  \label{equcd5}   \\
  \dot{\bm{\tilde{\theta}}}&=  \bm{\Lambda} \bm{g}^T (\bm{\xi}, t) ( \bm{\hat{\gamma}} - \bm{\gamma})\label{equcd6} \\
 \dot{\tilde{D}}&=  (Sgn (\bm{\hat{\gamma}}- \bm{\gamma} ))^T (\bm{\hat{\gamma}}- \bm{\gamma}).     \label{equcd7}
  \end{align}
\end{subequations}

It is noted that the right hand side of the closed-loop system \eqref{equcd} is discontinuous in $(\bm{\hat{\gamma}}- \bm{\gamma})$. Thus, the solution of the closed-loop system \eqref{equcd}  must  be defined in the Filipov sense \cite{cortes,shevitz1994lyapunov}.  Put the  closed-loop system \eqref{equcd} in the following compact form:
\begin{align} \label{eqc}
\dot{x}_{c}  = f_c (x_c, t).
\end{align}
Then,  the Filipov solution of \eqref{eqc} satisfies, for almost all $t \geq 0$,
\begin{align} \label{eqF}
 \dot{x}_{c}  \in K [f_c]  (x_c, t)
\end{align}
where $K [f_c] (x_c, t) $ is the Filipov set  of $f_c (x_c, t)$ \cite{cortes,shevitz1994lyapunov}. It is known that, for any scalar $x$,
the Filipov set  of $sgn (x)$  denoted by $K[sgn](x)$ is as follows \cite{shevitz1994lyapunov}:
\begin{align}
K[sgn](x)&=\begin{cases}
1,\quad &x> 0 \\
[-1,1],\quad &x= 0 \\
-1,\quad &x<0.
\end{cases} \label{sign_fun}
\end{align}
Thus, $x K[sgn](x)= |x|$.
Also, for any vector $x = \col (x_1, \cdots, x_N) \in \mathbb{R}^N$, let $K [sgn] (x) = \col ~(K [sgn] (x_1), \cdots, K [sgn] (x_N))$.
Then
\begin{align} \label{id1}
x^T K[sgn](x)&= \sum_{k=1}^N |x_i|.
\end{align}
For $\bm{\gamma} = \col (\gamma_1, \cdots, \gamma_N)$ with $\gamma_i \in  \mathbb{R}^{n_i}$, $i = 1, \cdots, N$, let
$$K [Sgn] (\bm{\gamma}) = \mbox{diag~} (K [sgn] (\gamma_1), \cdots, K [sgn] (\gamma_N)).$$

We are now ready to state our main result as follows.

\begin{thm} \label{thm3.3}
Under Assumptions \ref{ass3.1} to  \ref{ass6},
there exist $\delta^*>0$ such that,  for any $0<\delta<\delta^*$, and any initial condition,   the solution of the closed-loop system \eqref{equcd}  is bounded, and
\begin{subequations} \label{eq38dv}
\begin{align}
\lim_{t \rightarrow \infty} \bm{\gamma} (t) &= \bm{x}^*  \label{eq38dva} \\
\lim_{t \rightarrow \infty} \bm{x}^s (t) &= \bm{0}_{\sum_{i=1}^N n_i (r_i-1)}  \label{eq38dvb}  \\
\lim_{t \rightarrow \infty} \hat{\bm{\gamma}} (t) &= \bm{x}^*  \label{eq38dvc} \\
\lim_{t \rightarrow \infty} {\bm{z}} (t) &= {\bf 1} \otimes \bm{x}^*.   \label{eq38dvd}
\end{align}
\end{subequations}
As a result, $\lim_{t \rightarrow \infty} \bm{x} (t) = \bm{x}^*$.
\end{thm}

\begin{proof}
By Theorem \ref{th3.2},  there exist $\delta^*>0$ such that,  for any $0<\delta<\delta^*$, $(\bm{x}^*,   {\bf 1}_N \otimes \bm{x}^*)$ is the globally exponentially stable equilibrium of the subsystems (\ref{equcd3}) and (\ref{equcd5}).
Thus, (\ref{eq38dvc}) and (\ref{eq38dvd}) hold.
We only need to show that $\bm{\tilde{\theta}}$  and $\bm{\tilde{D}}$ are bounded and (\ref{eq38dva}) and  \eqref{eq38dvb} hold.
For this purpose, performing the coordinate transformation $\bm{\tilde{\gamma}}=  \bm{\hat{\gamma}} - \bm{\gamma}$ on (\ref{equcd1}), (\ref{equcd6}) and (\ref{equcd7}) gives  the following system:
\begin{subequations} \label{eqd39}
\begin{align}
\dot{\bm{\tilde{\gamma}}}&=  - \bm{g} (\bm{\xi},  t) \bm{\tilde{\theta}}-\kappa \bm{\tilde{\gamma}} - {Sgn} (\bm{\tilde{\gamma}})\hat{ D}   - \bm{d}   \label{eqd39a}  \\
\dot{\bm{\tilde{\theta}}}&=   \bm{\Lambda} \bm{g}^T(\bm{\xi},  t)  \bm{\tilde{\gamma}} \label{equd39c} \\
 \dot{\tilde{D}}&= (Sgn (\bm{\tilde{\gamma}}))^T \bm{\tilde{\gamma}}.     \label{eqd39e}
     \end{align}
\end{subequations}

Let  $\hat{\bm{z}}= \col (\bm{\tilde{\gamma}} , \bm{\tilde{\theta}}, {\tilde{D}}  )$ and denote the compact form of \eqref{eqd39} by $\dot{\hat{\bm{z}}} = \hat{f} (\hat{\bm{z}}, t)$.
Choose the Lyapunov function candidate for \eqref{eqd39} as follows:
\EQ \label{lyad}
\begin{aligned}
V (\hat{\bm{z}})=& \frac{1}{2} \bm{\tilde{\gamma}}^T  \bm{\tilde{\gamma}} +   \frac{1}{2} \bm{\tilde{\theta}}^T \bm{\Lambda}^{-1} \bm{\tilde{\theta}}
 +  \frac{1}{2} {\tilde{D}}^T {\tilde{D}}
 \end{aligned}
 \EN
whose gradient is $$\partial V= [\bm{\tilde{\gamma}}^T,   \bm{\tilde{\theta}}^T \bm{\Lambda}^{-1},  {\tilde{D}}^T ].$$
By Theorem 2.2 of \cite{shevitz1994lyapunov}, $\dot{V}$ exists almost everywhere (a.e.), and $\dot{V}\in ^{a.e.}\dot{\tilde{V}}$, where
\EQ\label{eq02}
\begin{aligned}
\dot{\tilde{V}}&=  \begin{array}{cc} \bigcap & \phi K [ \hat{f} ]  (\hat{\bm{z}}, t) \\
                                      \footnotesize{ \phi \in    \partial V}  & ~\end{array}   \\
& = - \bm{\tilde{\gamma}}^T \left (\bm{g} (\bm{\xi}, t) \bm{\tilde{\theta}}+\kappa \bm{\tilde{\gamma}} + K[ {Sgn}] (\bm{\tilde{\gamma}})\hat{ D} +\bm{d} \right )  \\
& +
\bm{\tilde{\theta}}^T \bm{g}^T (\bm{\xi}, t) \bm{\tilde{\gamma}} + {\tilde{D}}^T (K [Sgn] (\bm{\tilde{\gamma}}))^T \bm{\tilde{\gamma}}   \\
 & = - \bm{\tilde{\gamma}}^T \kappa \bm{\tilde{\gamma}}
  - \bm{\tilde{\gamma}}^T  ( K [{Sgn}] (\bm{\tilde{\gamma}})\hat{ D}  +\bm{d} ) \\
  & +{\tilde{D}}^T
(K [Sgn] (\bm{\tilde{\gamma}}))^T \bm{\tilde{\gamma}}   \\
&  = - \bm{\tilde{\gamma}}^T \kappa \bm{\tilde{\gamma}}  - \bm{\tilde{\gamma}}^T   K [{Sgn}] (\bm{\tilde{\gamma}}) { D}  - \bm{\tilde{\gamma}}^T \bm{d} \\
&  =  -  \bm{\tilde{\gamma}}^T \kappa \bm{\tilde{\gamma}} - \sum_{i=1}^N \left ({D}_i (\tilde{\gamma}_i)^T K [sgn] (\tilde{\gamma}_i) + d^T_i\tilde{\gamma}_i \right ).
\end{aligned}
\EN
Using (\ref{id1}) in \eqref{eq02} gives
\begin{align}
\dot{\tilde{V}}& =
  -  \bm{\tilde{\gamma}}^T \kappa \bm{\tilde{\gamma}} -\sum_{i=1}^N \sum_{j=1}^{n_i} ({D}_i |\tilde{\gamma}_{ij}| +  d_{ij}\tilde{\gamma}_{ij} ).
\end{align}
Thus, $\dot{\tilde{V}} = \dot{{V}}$. Noting $\sum_{j=1}^{n_i} ({D}_i |\tilde{\gamma}_{ij}| +  d_{ij}\tilde{\gamma}_{ij} ) \geq 0$ gives
\begin{align}
\dot{{V}}& \leq  - \bm{\tilde{\gamma}}^T \kappa \bm{\tilde{\gamma}}.
\end{align}
Thus, $\bm{\tilde{\gamma}}$ , $\bm{\tilde{\theta}}$, and $\bm{\tilde{D}}$   are all bounded.

Since \eqref{eq38dvc} implies  $\bm{\hat{\gamma}} (t) $ is bounded,  $\bm{{\gamma}} (t) $   is also bounded.
Since equation (\ref{eq36}) can be viewed as a stable linear differential equation in $x_i$ with an bounded input $\bm{\gamma}_i (t)$,
$\xi_i$ are all bounded. Thus, $\bm{\xi}$ is bounded.

Let \begin{align*}
	W(t)=\int_{0}^{t}   \bm{\tilde{\gamma}}^T (\tau) \kappa \bm{\tilde{\gamma}}  (\tau) d\tau.
\end{align*}
Then, for $t \geq 0$, $\dot{W} (t) = \bm{\tilde{\gamma}}^T (t) \kappa \bm{\tilde{\gamma}}  (t) \geq 0$ and
\begin{align*}
	W(t) \leq -  \int_{0}^{t} \dot{{V}}(\tau)d\tau=-{{V}}(t)+{{V}}(0).
\end{align*}
Since ${{V}}(t)$ is lower bounded, $\lim_{t \to \infty}W(t)$ has a finite limit.
 Moreover, we have
 \begin{equation}\label{ddotV22}
\ddot{W} (t)
=  - 2 \bm{\tilde{\gamma}}^T \kappa \left (\bm{g} (\bm{\xi},  t) \bm{\tilde{\theta}}+\kappa \bm{\tilde{\gamma}}  + {Sgn} (\bm{\gamma})\hat{ D} +\bm{d} \right).
\end{equation}
Since  we have shown that $\bm{\tilde{\gamma}}$ , $\bm{\tilde{\theta}}$, $\bm{\tilde{D}}$,  and $\bm{\xi}$   are all bounded,
 $\bm{g} (\bm{\xi},  t)$ is also bounded by Assumption \ref{ass6}.
Also, $\bm{d} (t)$ is bounded over  $[0, \infty)$ by Assumption \ref{ass6}.
Thus, if   $\bm{d} (t)$  is continuous over   $[0, \infty)$, 
$\ddot{W} $ is bounded and continuous over   $[0, \infty)$, which implies  $\dot{W} $ is uniformly continuous.
 By Babalat's Lemma \cite{Khalil},  $\lim_{t \rightarrow \infty} \dot{W}  (t) = 0$,  which implies
  $\lim_{t \rightarrow \infty} \bm{\tilde{\gamma}}  (t) =  \bm{0}_{n} $. 

Nevertheless, since we allow  $\bm{d} (t)$ to be piecewise continuous, $\ddot{W} $ may  be discontinuous at infinitely many time instants. Thus, we cannot  invoke  Barbalat's Lemma. 
 Nevertheless, denote the discontinuous time instants of $\bm{d}$ by $\tau_0 = 0, \tau_1, \cdots$. Then, for some $\tau >0$,   $\tau_{j+1}-\tau_j\geq \tau >0$ since $\bm{d}$ is piecewise continuous.
 Thus, we have shown that $W (t)$ satisfies the following three conditions:
\begin{itemize}
	\item [1)]
	$\lim_{t \to \infty} W (t)$ exists;
	\item [2)]
	$W (t)$ is twice differentiable on each time interval $[\tau_j,\tau_{j+1})$;  
	\item [3)]
	$\ddot{W}_i(t)$ is bounded over $[0, \infty)$ in the sense that $\|\ddot{W}\|_\infty$ is finite. 
\end{itemize}
Thus, by generalized Barbalat's Lemma as can be found in Corollary 2.5 of \cite{cshbook} to conclude $\lim_{t \to \infty} \dot{W} (t)=0$, and hence $\lim_{t\rightarrow \infty} \bm{\tilde{\gamma}}  (t) = 0$.
This fact together with the fact that $\lim_{t \rightarrow \infty} \bm{\hat{\gamma}} (t) =  \bm{x}^{*} $   implies that
$\lim_{t \rightarrow \infty} \bm{\gamma} (t) =  \bm{x}^{*}$.

Finally, consider equation (\ref{eq36}) which can be viewed as a stable linear differential equation in $x_i$ with the input $\bm{\gamma}_i (t)$ satisfying  $\lim_{t \rightarrow \infty} \bm{\gamma}_i (t) =  \bm{x}^{*}_i$. Thus, we have
$\lim_{t \rightarrow \infty} \bm{x}^s (t) =  \bm{0}_{\sum_{i=1}^N n_i (r_i-1)}$, which implies
$\lim_{t \rightarrow \infty} \bm{x} (t) =  \bm{x}^{*}$.

%
\end{proof}

\ifCLASSOPTIONcaptionsoff
  \newpage
\fi

\end{document}